\documentclass[11pt,a4paper,reqno]{amsart}

\usepackage{amsmath}
\usepackage{amsfonts}
\usepackage{a4wide}
\usepackage{amssymb}
\usepackage{amsthm}
\usepackage{mathrsfs}
\usepackage[colorlinks, citecolor=blue, linkcolor=red]{hyperref}

\theoremstyle{plain}

\newtheorem{teo}{Theorem}[section]
\newtheorem{lemma}[teo]{Lemma}
\newtheorem{prop}[teo]{Proposition}
\newtheorem{cor}[teo]{Corollary}
\newtheorem{ackn}{Acknowledgments\!}

\theoremstyle{definition}

\theoremstyle{remark}

\newtheorem{rem}[teo]{Remark}

\numberwithin{equation}{section}

\def\SS{{{\mathbb S}}}
\def\NN{{{\mathbb N}}}

\def\RR{{\mathbb R}}

\def\Ft{\mathcal{F}_{t}}
\def\Fts{\mathcal{F}_{t,s}}
\def\M1{\mathscr{M}_{1}}

\def\eps{\varepsilon}

\newcommand{\kn}{\mathbin{\bigcirc\mkern-15mu\wedge}}

\title[Critical metrics for quadratic functionals]{Some rigidity results on critical metrics for\\ quadratic functionals}

\author[Giovanni Catino]{Giovanni Catino}
\address[Giovanni Catino]{Dipartimento di Matematica, Politecnico di Milano, Piazza Leonardo da Vinci 32, 20133 Milano, Italy}
\email[]{giovanni.catino@polimi.it}

\begin{document}

\begin{abstract} In this paper we prove rigidity results on critical metrics for quadratic curvature functionals, involving the Ricci and the scalar curvature, on the space of Riemannian metrics with unit volume. It is well-known that Einstein metrics are always critical points. The purpose of this article is to show that, under some curvature conditions, a partial converse is true.  In particular,  for a class of quadratic curvature functionals, we prove that every critical metric with non-negative sectional curvature must be Einstein.
\end{abstract}

\maketitle

%\begin{center}
%
%\noindent{\it Key Words: critical metrics, quadratic functionals, rigidity results}
%
%\medskip
%
%\centerline{\bf AMS subject classification:  53C24, 53C25}
%
%\end{center}

%\tableofcontents

\section{Introduction}

\noindent Let $M^{n}$ be a closed smooth manifold of dimension $n\geq 3$ and let $\M1(M^{n})$ denote the space of equivalence classes of smooth Riemannian metrics of volume one on $M^{n}$.  As usual, metrics are identified if they are related by the action of the diffeomorphism group. To fix the notation, we recall the decomposition of the Riemann curvature tensor of a metric $g$ into the Weyl, Ricci and scalar curvature component
$$
Rm \,= \, W + \frac{1}{n-2} Ric \kn g - \frac{1}{(n-1)(n-2)}R \, g \kn g \,,
$$
where $\kn$ denotes the Kulkarni-Nomizu product. It is well known~\cite{hilbert} that Einstein metrics are critical points for the Einstein-Hilbert functional
$$
\mathcal{H} \,=\, \int R \,dV
$$
on $\M1(M^{n})$. From this perspective, it is natural to study  canonical metrics which arise as solutions of the Euler-Lagrange equations for more general curvature functionals.  In~\cite{berger3}, Berger commenced the study of Riemannian functionals which are quadratic in the curvature (see~\cite[Chapter 4]{besse} and~\cite{smot} for surveys). A basis for the space of quadratic curvature functionals is given by
$$
\mathcal{W} \,=\, \int |W|^{2} dV\, \quad \mathcal{\rho} \,=\, \int |Ric|^{2} dV\, \quad\, \mathcal{S} \,=\, \int R^{2} dV 
$$
and, from the decomposition of the Riemann tensor, one has
$$
\mathcal{R} \,=\, \int |Rm|^{2} dV \,=\, \int \left(|W|^{2}+\frac{4}{n-2}|Ric|^{2}-\frac{2}{(n-1)(n-2)}R^{2}\right) dV \,.
$$
We recall that, in dimension three, the Weyl curvature vanishes, and in dimension four, the Chern-Gauss-Bonnet formula
$$
\int \left(|W|^{2}-2|Ric|^{2}+\frac{2}{3}R^{2}\right) dV \, = \,32\pi^{2}\chi(M)\,
$$
implies that the Weyl functional $\mathcal{W}$ can be written as a linear combination of the other two (with the addition of a topological term). All such functionals, which also arise naturally as total actions in certain gravitational fields theories in physics, have been deeply studied in the last years by many authors. As it will be clear in Section~\ref{s-total}, in dimension greater than four, the Euler-Lagrange equation of the Weyl functional $\mathcal{W}$ has a different structure in the zero order curvature terms. In particular,  if $n>4$, it is not true that Einstein metrics are always critical points for this functional on $\M1(M^{n})$.

For these reasons, following the notation in~\cite{gurvia1}, in the first part of the paper, we will consider only the curvature functional
$$
\Ft \,=\, \int |Ric|^{2} dV +  t \int R^{2}  dV \,,
$$
defined for some constant $t\in \RR$ (with $t = -\infty$ formally corresponding to the functional $\mathcal{S}$). Since in dimensions greater than four $\Ft$ is not scale-invariant, it is natural to restrict the functional on $\M1(M^{n})$. Equivalently, one can consider a modified functional properly normalized with the volume of the manifold (see~\cite{gurvia1}).

It was already observed in~\cite{besse} that every Einstein metric is critical for $\Ft$ on $\M1(M^{n})$, for every $t\in\RR$. The first basic question is whether all critical metrics are necessarily Einstein. Of course, generally this is false. For instance, in dimension four, every Bach-flat metric is critical for $\mathcal{F}_{-1/3}$ and every Weyl and scalar flat metric is critical $\mathcal{F}_{-1/4}$ on $\M1(M^{4})$ (see~\cite[Chapter 4]{besse}). Moreover, Lamontagne in ~\cite{lamontagne1} constructed a homogeneous non-Einstein critical metric for $\mathcal{R}=4\mathcal{F}_{-1/4}$ on $\M1(\SS^{3})$. From this point of view, it is natural to ask under which conditions a critical metric for $\Ft$ must be Einstein. Typically, one assume some curvature conditions (of pointwise or integral type, positivity or negativity of the curvature, etc.) on the critical metric in order to prove rigidity properties. For instance, for $\mathcal{S}$ on $\M1(M^{3})$ in~\cite[Proposition 1.1]{ander2} the author assumed that the scalar curvature of the critical metric has definite sign (actually this holds in every dimension~\cite[Proposition 3.1]{cat2}); for $\mathcal{F}_{-1/3}$ on $\M1(M^{4})$ (variationally equivalent to $\mathcal{R}$) in~\cite{lamontagne2} it is proved that every critical metric with non-positive sectional curvature is Einstein; for $\mathcal{F}_{-1/3}$ on $\M1(M^{3})$ in~\cite{tanno} the author assumed a pointwise pinching condition on the Ricci curvature; finally, for $\mathcal{F}_{-3/8}$ on $\M1(M^{3})$ (variationally equivalent to the $\sigma_{2}$-functional) in~\cite{gurvia3} the authors proved that every critical metric must be Einstein (hence a space form), just assuming an integral condition, namely $\mathcal{F}_{-3/8} \leq 0$ (this result was extended in dimension greater than four in the locally conformally flat case in~\cite{huli}). As it will be clear from Corollary~\ref{cor-4} and Corollary~\ref{cor-5}, for some specific values of the parameter $t$, critical metrics for $\Ft$ inherit additional properties from the Euler-Lagrange equation which implies more constraints on the variational solution. For instance, every critical metric has constant scalar curvature if $t\neq -1/3$ and $n=4$ and it has constant $\sigma_{2}$-curvature if $t=-n/4(n-1)$ and $n\neq 4$. The results just quoted~\cite{lamontagne2, gurvia3, huli} belong to these cases. 

\medskip

In this paper we will prove some new rigidity results on critical metrics for $\Ft$ on $\M1(M^{n})$. Our first theorem characterizes critical metrics with non-negative sectional curvature.

\begin{teo} \label{t-main} 
Let $M^{n}$ be a closed manifold of dimension $n\geq 3$. If $g$ is a critical metric for $\Ft$ on $\M1(M^{n})$ for some $t< -1/2$ with non-negative sectional curvature, then $g$ is an Einstein metric.
\end{teo}
In the case $t=-1/2$, we can show the following result. 
\begin{teo}\label{t-unmezzo}
 Let $M^{n}$ be a closed manifold of dimension $n\geq 3$. If $g$ is a critical metric for $\mathcal{F}_{-1/2}$ on $\M1(M^{n})$  with non-negative sectional curvature, then $g$ is either Einstein or the following possibilities occur
\begin{itemize}

\item[(i)] If $n=3$, then the universal cover $(\tilde{M}^{3},\tilde{g})$ is isometric to $\big(\SS^{2}\times \RR, a\, g_{\SS^{2}}+ b\, g_{\RR}\big)$, for some positive constant $a,b>0$.

\item[(ii)] If $n=4$, then the universal cover $(\tilde{M}^{4},\tilde{g})$ is either isometric to $\big(\SS^{2}\times \SS^{2}, a\, g_{\SS^{2}}+ b\, g_{\SS^{2}}\big)$ or to $\big(\SS^{2}\times \RR^{2}, a\, g_{\SS^{2}}+ b\, g_{\RR^{2}}\big)$, for some positive constant $a,b>0$.

\item[(iii)] If $n>4$, then the universal cover $(\tilde{M}^{n},\tilde{g})$ is isometric to $\big(\SS^{2}\times \RR^{n-2}, a\, g_{\SS^{2}}+ b\, g_{\RR^{n-2}}\big)$, for some positive constant $a,b>0$.

\end{itemize}
\end{teo}

\begin{rem}
We notice that the condition $t < -1/2$ in Theorem~\ref{t-main} is optimal. In fact, for every $t > -1/2$, following~\cite{lamontagne1}, one can construct non-Einstein critical metrics $g_{t}$ for $\Ft$ on $\M1(\SS^{3})$ (see~\cite[Section 7]{gurvia1}). It turns out that these metrics have non-negative sectional curvature if $-1/2<t\leq 3/4$. On the other hand, a condition on the sectional curvature is necessary too. In fact, recently Gursky and Viaclovsky in~\cite{gurvia2} constructed critical metrics for $\Ft$ on $\M1(M^{4})$ for $t$ ``close'' to a given value which depends on the topology of the Einstein building blocks. In particular, they found solutions for $t$ close to $-1/2$ and in some cases for $t<-1/2$ (for precise estimates on the critical values see~\cite{via1}). As it is clear from the construction, all these metrics have changing sign sectional curvature.
\end{rem}

Concerning critical metrics with non-positive sectional curvature, we can extend Lamontagne result~\cite{lamontagne2} in dimension four, proving the following

\begin{teo} \label{t-neg} 
Let $M^{4}$ be a closed manifold of dimension four. If $g$ is a critical metric for $\Ft$ on $\M1(M^{4})$ for some $t\geq -1/4$ with non-positive sectional curvature, then $g$ is an Einstein metric.
\end{teo}

In Section~\ref{s-three} we provide a rigidity result on critical metrics for $\Ft$ on $\M1(M^{3})$ for $t>-1/2$.  As we have previously observed, one has to assume a stronger condition than non-negative sectional curvature to exclude the non-Einstein examples. Moreover, the estimates used in the proof of Theorem~\ref{t-main} are not sufficient in this regime of $t$, due to the presence of bad terms with the wrong sign. We were able to overcome this difficulties by ``weighting'' the Euler-Lagrange equation, and trying to compensate these quantities. Our result reads as follows

\begin{teo} \label{t-three}
Let $M^{3}$ be a closed manifold of dimension three. If $g$ is a critical metric for $\Ft$ on $\M1(M^{3})$ for some $-1/3 \leq t < -1/6$ with non-negative scalar curvature, then $g$ has constant positive sectional curvature if 
$$
|E|^{2} \,<\, \frac{(1+6t)^{2}}{24} \, R^{2} \,.
$$
\end{teo}

\begin{rem} 
In the cases $t=-1/3$ (trace-less Ricci functional) and $t=-5/16$ (Schouten functional), our result considerably improves the ones in~\cite{tanno} and in~\cite[Theorem 4.2]{udo}, respectively. Moreover, for $t=-1/4$ (Riemann functional), Theorem~\ref{t-three} gives a first answer to a basic question posed by Anderson~\cite[Section 6]{ander2} concerning the rigidity of critical metrics for the $L^{2}$-norm of the Riemann curvature tensor $\mathcal{R}$.
\end{rem}

\begin{rem} 
We notice that the result in Theorem~\ref{t-three} also holds when $t \leq -1/2$. In fact, the pinching assumption implies that $g$ has positive sectional curvature and Theorems~\ref{t-main} and~\ref{t-unmezzo} apply. Formally, in the case $t=-\infty$, we recover part of a result in~\cite{ander2} on critical metrics for $\mathcal{S}$.
\end{rem}

In the last part of the paper (Section~\ref{s-total}), we will compute the Euler-Lagrange equation satisfied by critical metrics for a general quadratic curvature functional of the form
$$
\Fts \,=\, \int |Ric|^{2} dV +  t \int R^{2}  dV + s \int |Rm|^{2} dV\,,
$$
defined for some constants $t,s\in\RR$. As we have already observed, from the variational point of view this functional differs from $\Ft$ only in dimension greater than four. As one would expect, space form metrics are critical for $\Fts$ on $\M1(M^{n})$, whereas, due to the presence of the full curvature tensor, in general Einstein metrics are not. Hence, a basic question would be to find variational characterization of space form metrics as critical points for $\Fts$ on $\M1(M^{n})$, in the spirit of the work of Gursky and Viaclovsky~\cite{gurvia3} on critical metrics for 
$$
\mathcal{F}_{-n/4(n-1)} \,=\, -2(n-2)^{2} \int \sigma_{2} (A) \,dV \,,
$$
where $\sigma_{2}(A)$ denotes the second elementary symmetric function of the Schouten tensor (see~\cite{via2}). As we have already observed (see Corollary~\ref{cor-5}), in dimension $n\neq 4$, a critical metric $g$ for this functional satisfy the Yamabe-type equation
$$
\sigma_{2} (A_{g}) \,=\, \hbox{const} \,.
$$
We will show in Corollay~\ref{c-yamabe} that, for every $t\in \RR$, the functional 
$$
\mathcal{F}_{t,-\frac{n+4(n-1)t}{4}}
$$
naturally extend the $\sigma_{2}$-functional (which correspond exactly to the case $s=0$). More precisely, we will prove that for every $t\in \RR$, a critical metric $g$ for $\mathcal{F}_{t,-\frac{n+4(n-1)t}{4}}$ on $\M1(M^{n})$, $n>4$, satisfies the Yamabe-type equation
$$
\big(1+2(n-1)t\big) \sigma_{2} (A_{g})-\frac{n+4(n-1)t}{16(n-2)} |W_{g}|^{2} \,=\, \hbox{const} \,.
$$
We hope that this property could help in proving some new variational characterizations of space forms as critical metrics of these functionals.

%For $t=s=-1/4$, this remarkable fact was already observed by Berger in~\cite[Section 7]{berger3}

\medskip

To conclude, we mention that the problem of finding conditions that guarantee rigidity of critical metrics for quadratic curvature functionals also has a lot of interest in the non-compact setting. For instance, Anderson in~\cite{ander3} proved that every complete three-dimensional critical metric for the Ricci functional $\mathcal{\rho}$ with non-negative scalar curvature is flat, whereas in~\cite{cat2} we showed a characterization of complete critical metrics for $\mathcal{S}$ with non-negative scalar curvature in every dimension.

\

\section{The Euler-Lagrange equation for $\Ft$} \label{s-euler}

In this section we will compute the Euler-Lagrange equation satisfied by critical metrics for $\Ft$ on $\M1(M^{n})$ (see also~\cite{gurvia1}). The gradients of the functionals $\mathcal{\rho}$ and $\mathcal{S}$ are given by (see~\cite[Proposition 4.66]{besse})
$$
(\nabla \mathcal{\rho})_{ij} \,=\, -\Delta R_{ij} - 2 R_{ikjl}R_{kl}+\nabla^{2}_{ij} R - \frac{1}{2}(\Delta R)g_{ij} + \frac{1}{2} |Ric|^{2} g_{ij} \,,
$$

$$
(\nabla \mathcal{S})_{ij} \,=\, 2 \nabla^{2}_{ij} R - 2(\Delta R) g_{ij} -2 R R_{ij} + \frac{1}{2} R^{2} g_{ij} \,.
$$
Hence, the gradient of $\Ft$ reads 
$$
(\nabla \Ft)_{ij} \,=\, -\Delta R_{ij} +(1+2t)\nabla^{2}_{ij} R - \frac{1+4t}{2}(\Delta R)g_{ij} + \frac{1}{2} \Big( |Ric|^{2}+ t R^{2}\Big) g_{ij} - 2 R_{ikjl}R_{kl} -2t R R_{ij} \,.
$$
Moreover, $g$ is critical for $\Ft$ on $\M1(M^{n})$ if and only if $(\nabla \Ft) = c \, g$, for some Lagrange multiplier $c\in \RR$ (see~\cite{besse}). Tracing this equation, we obtain
$$
\frac{n-4}{2}\Big(|Ric|^{2}+ t R^{2}\Big)-\frac{n+4(n-1)t}{2} \Delta R \,=\, nc \,.
$$
From these, we get that $g$ is critical for $\Ft$ on $\M1(M^{n})$ if and only if
\begin{equation}\label{eq-ric}
-\Delta R_{ij} + (1+2t) \nabla^{2}_{ij} R -\frac{2t}{n}(\Delta R) g_{ij} + \frac{2}{n}\Big(|Ric|^{2}+ t R^{2}\Big)g_{ij}- 2 R_{ikjl}R_{kl} -2t R R_{ij} \,=\, 0\,,
\end{equation}
and
$$
\Big(n+4(n-1)t\Big)\Delta R \,=\, (n-4) \Big( |Ric|^{2} + t R^{2} - \lambda \Big) \,,
$$
where $\lambda = \Ft (g)$. Defining the tensor $E$ to be the trace-less Ricci tensor, $E_{ij} = R_{ij} - \frac{1}{n} R g_{ij}$, we obtain the Euler-Lagrange equation of critical metrics for $\Ft$ on $\M1(M^{n})$.

\begin{prop}\label{p-euler}
Let $M^{n}$ be a closed manifold of dimension $n\geq 3$. A metric $g$ is critical for $\Ft$ on $\M1(M^{n})$ if and only if it satisfies the following equations

\begin{equation}\label{eq1}
\Delta E_{ij} \,=\,  (1+2t) \nabla^{2}_{ij} R - \frac{1+2t}{n} (\Delta R) g_{ij} - 2 R_{ikjl} E_{kl} -\frac{2+2nt}{n} R E_{ij} + \frac{2}{n} |E|^{2} g_{ij} \,,
\end{equation}

\begin{equation}\label{eq2}
\Big(n+4(n-1)t\Big)\Delta R \,=\, (n-4) \Big( |Ric|^{2} + t R^{2} - \lambda \Big) \,,
\end{equation}
where $\lambda = \Ft (g)$.
\end{prop}

In particular, it follows that Einstein metrics are critical (see~\cite[Corollary 4.67]{besse}).
\begin{cor} Any Einstein metric is critical for $\Ft$ on $\M1(M^{n})$.
\end{cor}

From equation~\eqref{eq2}, if $n=4$ and $t\neq -1/3$, we immediately get the following result.

\begin{cor}\label{cor-4}
Let $M^{4}$ be a closed manifold of dimension four. If $g$ is a critical metric for $\Ft$ on $\M1(M^{4})$ for some $t\neq -1/3$, then $g$ has constant scalar curvature.
\end{cor}
Notice that, in dimension four, Gauss-Bonnet formula implies that $\mathcal{F}_{-1/3}$ is proportional (plus a constant term) to the Weyl functional $\mathcal{W}$. Hence, critical metrics are Bach-flat and, in general, do not have constant scalar curvature. On the other hand, if $t=-n/4(n-1)$, one has
$$
|Ric|^{2}-\frac{n}{4(n-1)} R^{2} \,=\, -2(n-2)^{2} \sigma_{2} (A) \,,  
$$
where $\sigma_{2}(A)$ denotes the second elementary symmetric function of the Schouten tensor
$$
A \,=\, \frac{1}{n-2} \Big( Ric - \frac{1}{2(n-1)} R\, g \Big) \,.
$$
Hence, when $n\neq4$ and $t=-n/4(n-1)$, we have
$$
\mathcal{F}_{-n/4(n-1)} \,=\, -2(n-2)^{2} \int \sigma_{2} (A) \,dV 
$$
and equation~\eqref{eq2} implies the following
\begin{cor} \label{cor-5}
Let $M^{n}$ be a closed manifold of dimension $n\neq 4$. If $g$ is a critical metric for $\mathcal{F}_{-n/4(n-1)}$ on $\M1(M^{n})$, then $g$ has constant $\sigma_{2}$-curvature.
\end{cor}

Now, contracting equation~\eqref{eq1} with $E$ we obtain the following Weitzenb\"ock formula.

\begin{prop}
Let $M^{n}$ be a closed manifold of dimension $n$. If $g$ is a critical metric for $\Ft$ on $\M1(M^{n})$, then the following formula holds

\begin{equation}\label{eq-w}
\frac{1}{2}\Delta |E|^{2} \,=\,  |\nabla E|^{2} + (1+2t) E_{ij}\nabla^{2}_{ij} R - 2 R_{ikjl} E_{ij} E_{kl} -\frac{2+2nt}{n} R |E|^{2} \,.
\end{equation}
\end{prop}

\begin{cor}\label{cor-int}
Let $M^{n}$ be a closed manifold of dimension $n$. If $g$ is a critical metric for $\Ft$ on $\M1(M^{n})$, then
\begin{equation*}
\int \left(|\nabla E|^{2} - \frac{(n-2)(1+2t)}{2n}|\nabla R|^{2} \right) dV \,=\, 2 \int \left( R_{ikjl} E_{ij} E_{kl} + \frac{1+nt}{n} R|E|^{2} \right
) dV \,.
\end{equation*}
\end{cor}
\begin{proof} We simply integrate by parts equation~\eqref{eq-w} and use the second Bianchi identity 
$$
\nabla_{i} E_{ij} \,=\, \nabla_{i} R_{ij} - \frac{1}{n} \nabla_{j}R \,=\, \frac{n-2}{2n} \nabla_{j} R \,.
$$
\end{proof}

\

\section{Critical metrics with non-negative sectional curvature}

In this section we will prove Theorem~\ref{t-main} and Theorem~\ref{t-unmezzo}. The first key observation is the following pointwise estimate which is satisfied by every metric with non-negative sectional curvature.
\begin{prop}\label{p-est}
Let $(M^{n},g)$ be a Riemannian manifold of dimension $n\geq 3$ with non-negative sectional curvature. Then, the following estimate holds
$$
R_{ikjl} E_{ij} E_{kl} \leq \frac{n-2}{2n} R|E|^{2} \,.
$$
\end{prop}
\begin{proof} Let $\{e_{i}\}$, $i=1,\ldots, n$, be the eigenvectors of $E$ and let $\lambda_{i}$ be the corresponding eigenvalues. Moreover, let $\sigma_{ij}$ be the sectional curvature defined by the two-plane spanned by $e_{i}$ and $e_{j}$. We want to prove that the quantity
$$
R_{ikjl} E_{ij} E_{kl}  - \frac{n-2}{2n} R|E|^{2} \,=\, \sum_{i,j=1}^{n} \lambda_{i}\lambda_{j} \sigma_{ij} - \frac{n-2}{2n} R \sum_{k=1}^{n} \lambda_{k}^{2}  
$$
is non-positive if $\sigma_{ij} \geq 0$ for all $i,j=1,\ldots,n$. The scalar curvature can be written as
$$
R \,=\, g^{ij}g^{kl}R_{ikjl} \,=\, \sum_{i,j=1}^{n} \sigma_{ij} \,=\, 2 \sum_{i<j}\sigma_{ij}\,.
$$
Hence, one has the following
\begin{eqnarray*} 
\sum_{i,j=1}^{n} \lambda_{i}\lambda_{j} \sigma_{ij} - \frac{n-2}{2n}  R \sum_{k=1}^{n} \lambda_{k}^{2} &=& 2 \sum_{i<j}\lambda_{i}\lambda_{j} \sigma_{ij} - \frac{n-2}{n}  \sum_{i<j} \sigma_{ij} \sum_{k=1}^{n} \lambda_{k}^{2} \\
&=& \sum_{i<j} \Big(2\lambda_{i}\lambda_{j} -\frac{n-2}{n}\sum_{k=1}^{n} \lambda_{k}^{2} \Big) \sigma_{ij}\,. 
\end{eqnarray*}
On the other hand, one has 
$$
\sum_{k=1}^{n} \lambda_{k}^{2} \,=\, \lambda_{i}^{2} + \lambda_{j}^{2} + \sum_{k\neq i,j}\lambda_{k}^{2}\,.
$$
Moreover, using the Cauchy-Schwarz inequality and the fact that $\sum_{k=1}^{n} \lambda_{k} =0$, we obtain
$$
\sum_{k\neq i,j}\lambda_{k}^{2} \,\geq \, \frac{1}{n-2} \Big( \sum_{k\neq i,j}\lambda_{k}\Big)^{2} \,=\, \frac{1}{n-2} \big( \lambda_{i} + \lambda_{j} \big)^{2} \,.
$$
Hence, the following estimate holds 
$$
\sum_{k=1}^{n} \lambda_{k}^{2} \,\geq \, \frac{n-1}{n-2} \big(\lambda_{i}^{2} + \lambda_{j}^{2} \big) + \frac{2}{n-2} \lambda_{i} \lambda_{j}\,.
$$
Using this, since $\sigma_{ij} \geq 0$, it follows that
\begin{eqnarray*} 
\sum_{i,j=1}^{n} \lambda_{i}\lambda_{j} \sigma_{ij} - \frac{n-2}{2n} R \sum_{k=1}^{n} \lambda_{k}^{2} &\leq& \frac{n-1}{n}\sum_{i<j} \big(2\lambda_{i}\lambda_{j} -\big(\lambda_{i}^{2} + \lambda_{j}^{2} \big) \Big) \sigma_{ij} \\
&=&- \frac{n-1}{n}\sum_{i<j} (\lambda_{i}-\lambda_{j})^{2} \sigma_{ij}\, \leq \, 0 \,.
\end{eqnarray*}
This concludes the proof of the proposition.
\end{proof}

From Corollary~\ref{cor-int}, we have that if $g$ is a critical metric for $\Ft$ on $\M1(M^{n})$, then
\begin{equation*}
\int \left(|\nabla E|^{2} - \frac{(n-2)(1+2t)}{2n}|\nabla R|^{2} \right) dV \,=\, 2 \int \left( R_{ikjl} E_{ij} E_{kl} + \frac{1+nt}{n} R|E|^{2} \right
) dV\,.
\end{equation*}
If $t\leq -1/2$, the left-hand side is nonnegative and is zero if and only if $|\nabla E|=0$. On the other hand, if $g$ has non-negative sectional curvature, Proposition~\ref{p-est}, implies that 
$$
R_{ikjl} E_{ij} E_{kl} + \frac{1+nt}{n} R|E|^{2} \, \leq \, \Big(t+\frac{1}{2}\Big)R|E|^{2}\,.
$$
In particular, if $g$ is a critical metric $g$ for $\Ft$ on $\M1(M^{n})$ with non-negative sectional curvature, we have
\begin{equation*}
\int \left(|\nabla E|^{2} - \frac{(n-2)(1+2t)}{2n}|\nabla R|^{2} \right) dV \,\leq\, (1+2t) \int R|E|^{2}  dV\,.
\end{equation*}
Hence, if $t<-1/2$, then $g$ has to be scalar flat or Einstein. If $R\equiv 0$, then $g$ must be flat, since it has non-negative sectional curvature. This concludes the proof of Theorem~\ref{t-main}.

\

If $t=-1/2$, we have that $\nabla E=0$ on $M^{n}$. In particular, $g$ has constant scalar curvature, and, from the de Rham decomposition theorem, is locally a product of Einstein metrics. Again, if $R\equiv 0$, then $g$ has to be flat. So, from now on, we will assume that $g$ has constant positive scalar curvature $R>0$. From the critical equation~\eqref{eq-ric} we get
$$
\frac{1}{n}\Big(|Ric|^{2} -\frac{1}{2} R^{2}\Big)g_{ij}-  R_{ikjl}R_{kl} +\frac{1}{2} R R_{ij} \,=\, 0
$$
Moreover, since $g$ has parallel Ricci tensor, the commutation rule of covariant derivatives, implies
$$
0 \,=\, \nabla_{p} \nabla_{j} R_{ik} - \nabla_{j} \nabla_{p} R_{ik} \,=\, R_{lijp} R_{kl} + R_{lkjp} R_{il}\,.
$$
Tracing with $g^{pk}$, we get $R_{ikjl}R_{kl} = R_{ik} R_{jk}$. Hence, the Ricci tensor satisfies the quadratic condition 
$$
\frac{1}{n}\Big(2|Ric|^{2} - R^{2}\Big)g_{ij}-  2 R_{ik}R_{jk} + R R_{ij} \,=\, 0
$$
In particular, every eigenvalue $\mu$ of the Ricci tensor, satisfies the equation
$$
\frac{1}{n}\Big(2|Ric|^{2} - R^{2}\Big)-  2\mu^{2} + R \mu \,=\, 0 \,,
$$
which implies that 
$$
2\mu^{2}-R\mu-\frac{2}{n} \Big( |E|^{2} - \frac{n-2}{2n}R^{2}\Big) \,=\, 0 \,.
$$
Solving this equation, we get that every eigenvalue satisfies
$$
\mu \,=\, \frac{1}{4}R \pm  \sqrt{\frac{(n-4)^{2}}{16n^{2}} R^{2} + \frac{1}{n}|E|^{2}} \,.
$$
Now, let us assume that for some $\NN \ni m\geq 1$, we have
$$
\mu_{1}=\ldots=\mu_{m} \,=\, \frac{1}{4}R + \sqrt{\frac{(n-4)^{2}}{16n^{2}} R^{2} + \frac{1}{n}|E|^{2}}
$$
and
$$
\mu_{m+1}=\ldots=\mu_{n} \,=\, \frac{1}{4}R - \sqrt{\frac{(n-4)^{2}}{16n^{2}} R^{2} + \frac{1}{n}|E|^{2}}
$$
Clearly, if $m=1$, we have that $\mu_{1}=0$ on $M^{n}$ and $\mu_{2}=\ldots=\mu_{n}=\frac{1}{2}R$. If this is the case, then $n=3$ and we are exactly in case (i) of Theorem~\ref{t-unmezzo}. On the other hand, if $m=n$ then the metric is Einstein. So, from now on we will assume that $2\leq m < n$. Moreover, by summing all the eigenvalues, we get the identity
$$
\frac{n-4}{4} R \,=\, (n-2m) \sqrt{\frac{(n-4)^{2}}{16n^{2}} R^{2} + \frac{1}{n}|E|^{2}} \,.
$$ 
Since $R>0$, we have that $m=2$ if $n=3$ or $n=4$, and $m<n/2$ if $n>4$. Thus, if $n=4$, we have that
$$
\mu_{1} = \mu_{2} \,=\, \frac{1}{4} R  + \frac{1}{2}|E| \quad\quad \hbox{and}\quad \quad \mu_{3} = \mu_{4} \,=\, \frac{1}{4} R  - \frac{1}{2}|E| \,.
$$
On the other hand, if $n\neq 4$, one has
$$
|E|^{2} \,=\, \frac{m(n-m)(n-4)^{2}}{4n(n-2m)^{2}}R^{2} \,.
$$
Thus, if $n=3$, we have proved that
$$
\mu_{1} = \mu_{2} \,=\, \frac{1}{2} R \quad\quad \hbox{and}\quad \quad \mu_{3} \,=\, 0 \,,
$$
whereas, for $n>4$, one has that
$$
\mu_{1}=\ldots=\mu_{m} \,=\, \frac{n-m-2}{2(n-2m)} R \quad\quad \hbox{and}\quad \quad \mu_{m+1}=\ldots=\mu_{n} \,=\, \frac{2-m}{2(n-2m)} R \,.
$$
Since $g$ has non-negative Ricci curvature, if $n>4$ the only admissible case is $m=2$, so
$$
\mu_{1}=\mu_{2} \,=\, \frac{1}{2} R \quad\quad \hbox{and}\quad \quad \mu_{3}=\ldots=\mu_{n} \,=\, 0 \,.
$$
In conclusion, as a consequence of the de Rham decomposition theorem, we have shown that if $g$ is a critical metric for $\mathcal{F}_{-1/2}$ on $\M1(M^{n})$ with non-negative sectional curvature, then $g$ is either Einstein or the following possibilities occur

\begin{itemize}

\item[(i)] If $n=3$, then the eigenvalues of the Ricci tensor are equal to $\mu_{1} = \mu_{2} = \frac{1}{2} R $ and $\mu_{3} = 0$, where $R$ is a positive constant. In particular, the universal cover $(\tilde{M}^{3},\tilde{g})$ is isometric to $\big(\SS^{2}\times \RR, a\, g_{\SS^{2}}+ b\, g_{\RR}\big)$, for some positive constant $a,b>0$.

\item[(ii)] If $n=4$, then the eigenvalues of the Ricci tensor are equal to $\mu_{1} = \mu_{2} = \frac{1}{4} R  + \frac{1}{2}|E|$ and $\mu_{3} = \mu_{4} = \frac{1}{4} R  - \frac{1}{2}|E|$, where $R$ and $|E|$ are positive constants. Now, there are two possibilities: either $\mu_{3}=\mu_{4}>0$ or $\mu_{3}=\mu_{4}=0$. Hence, the universal cover $(\tilde{M}^{4},\tilde{g})$ is either isometric to $\big(\SS^{2}\times \SS^{2}, a\, g_{\SS^{2}}+ b\, g_{\SS^{2}}\big)$ or to $\big(\SS^{2}\times \RR^{2}, a\, g_{\SS^{2}}+ b\, g_{\RR^{2}}\big)$, for some positive constant $a,b>0$.

\item[(iii)] If $n>4$, then the eigenvalues of the Ricci tensor are equal to $\mu_{1} = \mu_{2} = \frac{1}{2} R $ and $\mu_{3} = \ldots = \mu_{n} = 0$, where $R$ is a positive constant. Since $g$ has non-negative sectional curvature, the Ricci flat part has to be flat. Hence, the universal cover $(\tilde{M}^{n},\tilde{g})$ is isometric to $\big(\SS^{2}\times \RR^{n-2}, a\, g_{\SS^{2}}+ b\, g_{\RR^{n-2}}\big)$, for some positive constant $a,b>0$.

\end{itemize}

This concludes the proof of Theorem~\ref{t-unmezzo}.

\

\section{Critical metrics with non-positive sectional curvature}

In this section we will prove Theorem~\ref{t-neg}. First of all we show some useful estimates which hold for every $n$-dimensional Riemannian manifold. We recall the definition of the Cotton tensor
$$
C_{ijk} \,=\, \nabla_{k} R_{ij} - \nabla_{j} R_{ik} -
\frac{1}{2(n-1)} \big( \nabla_{k} R \, g_{ij} - \nabla_{j} R \,
g_{ik} \big)\,.
$$
We have the following formula (see~\cite[Section 4]{gurvia3} for this formula in dimension three).
\begin{prop}\label{pro-cotton}
Let $(M^{n},g)$ be a Riemannian manifold of dimension $n\geq 3$. Then, the following integral formula holds
$$
\int \Big(|\nabla E|^{2} - \frac{(n-2)^{2}}{4n(n-1)}|\nabla R|^{2} -\frac{1}{2} |C|^{2} \Big) dV \,=\, \int \Big( R_{ikjl} E_{ij}E_{kl} - E_{ij}E_{ik}E_{jl}-\frac{1}{n} R|E|^{2} \Big) dV \,.
$$
\end{prop}
\begin{proof} A simple computation shows that
\begin{equation} \label{eqcotton}
\frac{1}{2}|C|^{2} \,=\, |\nabla E|^{2} - \frac{(n-2)^{2}}{4n^{2}(n-1)}|\nabla R|^{2} - \nabla_{k} E_{ij} \nabla_{j}E_{ik} \,.
\end{equation}
Integrating by parts the last term, we get
\begin{eqnarray*}
\int \nabla_{k} E_{ij} \nabla_{j}E_{ik}\,dV &=& -\int E_{ij} \nabla_{k} \nabla_{j} E_{ik} \, dV\\
&=& - \int \Big( E_{ij} \nabla_{j} \nabla_{k} E_{ik} + R_{kjil}E_{ij}E_{kl} + E_{ij}E_{ik}E_{jl} + \frac{1}{n} R|E|^{2}\Big) dV \\
&=& - \int \Big( \frac{n-2}{2n} E_{ij} \nabla_{i}\nabla_{j} R - R_{ikjl}E_{ij}E_{kl} + E_{ij}E_{ik}E_{jl} + \frac{1}{n} R|E|^{2}\Big) dV \\
&=&  \int \Big( \frac{(n-2)^{2}}{4n^{2}} |\nabla R|^{2} + R_{ikjl}E_{ij}E_{kl} - E_{ij}E_{ik}E_{jl} - \frac{1}{n} R|E|^{2}\Big) dV \,.
\end{eqnarray*} 
Using this identity and integrating equation~\eqref{eqcotton} we get the desired result.
\end{proof}

\begin{prop}\label{pro-est2}
Let $(M^{n},g)$ be a Riemannian manifold of dimension $n\geq 3$ with non-positive sectional curvature. Then, the following estimate holds
$$
R_{ikjl} E_{ij} E_{kl} +\frac{1}{n} R|E|^{2} + E_{ij}E_{ik}E_{jl} \,\leq\, 0\,.
$$
Moreover, if $n=4$, equality occurs if and only if $|E|=0$.
\end{prop}
\begin{proof} As in the proof of Proposition~\ref{p-est}, we let $\{e_{i}\}$, $i=1,\ldots, n$, be the eigenvectors of $E$ and $Ric$ and let $\lambda_{i}$ and $\mu_{i}$ be the corresponding eigenvalues. Moreover, let $\sigma_{ij}$ be the sectional curvature defined by the two-plane spanned by $e_{i}$ and $e_{j}$. We want to prove that the quantity
$$
R_{ikjl} E_{ij} E_{kl} +\frac{1}{n} R|E|^{2} + E_{ij}E_{ik}E_{jl} \,=\, \sum_{i,j=1}^{n} \lambda_{i}\lambda_{j} \sigma_{ij} + \frac{1}{n} R \sum_{k=1}^{n} \lambda_{k}^{2} + \sum_{k=1}^{n}\lambda_{k}^{3}
$$
is non-positive if $\sigma_{ij} \leq 0$ for all $i,j=1,\ldots,n$. 
First of all, we notice that
$$
\sum_{i,j=1}^{n} \lambda_{i}\lambda_{j} \sigma_{ij} + \frac{1}{n} R \sum_{k=1}^{n} \lambda_{k}^{2} + \sum_{k=1}^{n}\lambda_{k}^{3} \,=\, 2\sum_{i<j}^{n} \lambda_{i}\lambda_{j} \sigma_{ij} + \sum_{k=1}^{n}\mu_{k} \lambda_{k}^{2} \,,
$$
since $\mu_{k}=\lambda_{k}+\frac{1}{n}R$. Moreover, for every $k$, $\mu_{k}=\sum_{i\neq k} \sigma_{ik}$ and
$$
\sum_{k=1}^{n}\mu_{k} \lambda_{k}^{2} \,=\, \sum_{i<j} (\lambda_{i}^{2}+\lambda_{j}^{2})\sigma_{ij} \,.
$$ 
Hence
$$
\sum_{i,j=1}^{n} \lambda_{i}\lambda_{j} \sigma_{ij} + \frac{1}{n} R \sum_{k=1}^{n} \lambda_{k}^{2} + \sum_{k=1}^{n}\lambda_{k}^{3} \,=\, \sum_{i<j} \big( 2 \lambda_{i}\lambda_{j} + \lambda_{i}^{2} + \lambda_{j}^{2} \big) \sigma_{ij} \,=\, \sum_{i<j} (\lambda_{i}+ \lambda_{j})^{2} \sigma_{ij} \leq 0 \,,
$$
since $\sigma_{ij}\leq 0$, and the inequality is proved.

Now we will analyze the equality case in dimension four. If equality occurs, then, choosing $(i,j,k,l)$ as a permutation of $(1,2,3,4)$, we get

\begin{eqnarray*}
0 &=& \sum_{i,j=1}^{4} (\lambda_{i}+ \lambda_{j})^{2} \sigma_{ij} =  \sum_{i,j=1}^{4} \Big(\mu_{i}+\mu_{j}-\frac{1}{2}R\Big)^{2} \sigma_{ij} \\ 
&=& \sum_{i,j=1}^{4} \Big(\frac{\mu_{i}+\mu_{j}}{2}-\frac{\mu_{k}+\mu_{l}}{2}\Big)^{2} \sigma_{ij} \\
&=& \sum_{i,j=1}^{4} (\sigma_{ij}-\sigma_{kl})^{2} \sigma_{ij} \\
&=& (\sigma_{12}-\sigma_{34})^{2}(\sigma_{12}+\sigma_{34})+(\sigma_{13}-\sigma_{24})^{2}(\sigma_{13}+\sigma_{24})+(\sigma_{14}-\sigma_{23})^{2}(\sigma_{14}+\sigma_{23})\,.
\end{eqnarray*}
This implies that $\sigma_{12}=\sigma_{34}$, $\sigma_{13}=\sigma_{24}$ and $\sigma_{14}=\sigma_{23}$, since $g$ has non-positive sectional curvatures. Now, if we compute $\lambda_{1}$,  we get
$$
\lambda_{1} \,=\, \mu_{1} - \frac{1}{4}R \,=\, \sum_{j=2}^{4}\sigma_{1j} - \frac{1}{2}\sum_{i<j}\sigma_{ij} \,=\, 0 \,.
$$
A similar argument shows that $\lambda_{i}=0$ for every $i$, so the metric must be Einstein.
\end{proof}

Now we can prove Theorem~\ref{t-neg}. Let $M^{4}$ be a compact manifold of dimension four and $g$ be a critical metric for $\Ft$ on $\M1(M^{4})$ for some $t\geq -1/4$ with non-positive sectional curvature. By Corollary~\ref{cor-4}, we know that $g$ has constant scalar curvature. Moreover, since $g$ is critical, then, Corollary~\ref{cor-int} implies that
$$
\int |\nabla E|^{2} dV \,=\, 2 \int \left( R_{ikjl} E_{ij} E_{kl} + \frac{1+4t}{4} R|E|^{2} \right
) dV \,\leq \, 2 \int R_{ikjl} E_{ij} E_{kl} \,dV \,,
$$
since $R\leq 0$ and $t\geq 1/4$.
Using Propositions~\ref{pro-cotton}, we obtain
$$
\frac{1}{2}\int |C|^{2} dV \,\leq\, \int  \left( R_{ikjl} E_{ij} E_{kl} +\frac{1}{n} R|E|^{2} + E_{ij}E_{ik}E_{jl} \right) dV \,.
$$ 
From Proposition~\ref{pro-est2}, it follows that the right hand side is non-positive, so it must be zero and the metric must be Einstein. This concludes the proof of Theorem~\ref{t-neg}.

\

\section{Three dimensional critical metrics with positively pinched curvature} \label{s-three}

In this section we will prove Theorem~\ref{t-three}. Let $M^{3}$ be a closed manifold of dimension three and $g$ be a critical metric for $\Ft$ on $\M1(M^{3})$ for some $-1/3 \leq t < -1/6$. We will assume that $g$ has non-negative scalar curvature and it satisfies the piching condition
\begin{equation} \label{pinch3d}
|E|^{2} \,<\, \frac{(1+6t)^{2}}{24} \, R^{2} \,.
\end{equation}
To prove that $g$ must be Einstein, i.e. a constant sectional curvature metric, one would like to follow the proof of Theorem~\ref{t-main}. Unfortunately, it is easy to observe that, if $t>-1/2$, the left-hand side of the integral formula in Corollary~\ref{cor-int} could be non-positive. This observation leads us to find a different strategy in order to deal with the gradient terms of the Euler-Lagrange equation. In dimension three, the Riemann curvature tensor decomposes as
$$
R_{ikjl} \,=\, E_{ij}g_{kl}-E_{il}g_{jk}+E_{kl}g_{ij}-E_{kj}g_{il} + \frac{1}{6}\, R \,(g_{ij}g_{kl}-g_{il}g_{jk}) \,.
$$
Hence, the equation~\eqref{eq-w} for critical metrics reads 
$$
\frac{1}{2}\Delta |E|^{2} \,=\,  |\nabla E|^{2} + (1+2t) E_{ij}\nabla^{2}_{ij} R + 4 E_{ij} E_{ik}E_{kj} -\frac{1+6t}{3} R |E|^{2} \,.
$$
Multiplying this latter with the scalar curvature $R$ and integrating by parts over $M^{3}$, we obtain
\begin{eqnarray}\label{eq-127}
&&\int \left( \frac{1}{2}\langle \nabla |E|^{2}, \nabla R \rangle + R|\nabla E|^{2} - \frac{1+2t}{6}R|\nabla R|^{2} - (1+2t)E(\nabla R, \nabla R) \right) dV \\\nonumber
&& \hspace{5.5cm}=\, \int \left( \frac{1+6t}{3} R^{2} |E|^{2} - 4 R E_{ij} E_{ik}E_{kj} \right)dV \,,
\end{eqnarray}
where we used the second Bianchi identity $\nabla_{i} E_{ij} =  \nabla_{j} R / 6$. First of all we observe that, under the assumption of Theorem~\ref{t-three}, the right-hand side is non-positive. In fact, since $E$ is a symmetric traceless two-tensor, then one has the following sharp inequality~\cite[Lemma 4.2]{gurvia3}
$$
|E_{ij} E_{ik}E_{kj}| \, \leq \, \frac{1}{\sqrt{6}} |E|^{3} \,.
$$
Thus, since $R > 0$, from the pinching assumption~\eqref{pinch3d}, we get
\begin{equation}\label{eq-128}
\frac{1+6t}{3} R^{2} |E|^{2} - 4 R E_{ij} E_{ik}E_{kj} \,\leq\, R|E|^{2} \left( \frac{1+6t}{3}R - \frac{4}{\sqrt{6}}|E| \right) \, \leq \, 0 \,.
\end{equation}
To conclude the proof, we have to estimate the left-hand side of equation~\eqref{eq-127}. We have the following

\begin{lemma}
Under the assumption of Theorem~\ref{t-three}, one has
$$
\int \left( \frac{1}{2}\langle \nabla |E|^{2}, \nabla R \rangle + R|\nabla E|^{2} - \frac{1+2t}{6}R|\nabla R|^{2} - (1+2t)E(\nabla R, \nabla R) \right) dV \, \geq \, 0 \,.
$$ 
\end{lemma}
\begin{proof} 
From Bochner formula, one has
\begin{eqnarray*}
\frac{1}{2} \Delta|\nabla R|^{2} &=& |\nabla^{2} R|^{2} + Ric(\nabla R, \nabla R) + \langle \nabla \Delta R, \nabla R  \rangle \\
&=&  |\nabla^{2} R|^{2} + E(\nabla R, \nabla R) + \frac{1}{3}R |\nabla R|^{2} + \langle \nabla \Delta R, \nabla R  \rangle \,.
\end{eqnarray*}
Integrating by parts, we get
\begin{eqnarray*}
\int E(\nabla R, \nabla R) \, dV &=& - \int \left(|\nabla^{2} R |^{2} -|\Delta R|^{2} + \frac{1}{3} R |\nabla R|^{2} \right) dV \\
&\leq &  \int \left(\frac{2}{3}|\Delta R|^{2} - \frac{1}{3} R |\nabla R|^{2} \right) dV  \,,
\end{eqnarray*}
where we have used the algebraic inequality $|\nabla^{2} R|^{2} \geq |\Delta R|^{2} / 3$. On the other hand, from the traced equation of critical metrics 
\begin{equation}\label{eq-trace}
\Delta R \,=\, -\frac{1}{3+8t} \Big( |Ric|^{2} + t R^{2} - \lambda \Big) \,,
\end{equation} 
one has
\begin{eqnarray*}
\int |\Delta R |^{2} dV &=& \int R \Delta^{2} R dV \,\,=\,\, -\frac{1}{3+8t} \int R \Delta |Ric|^{2}dV - \frac{t}{3+8t} \int R \Delta R^{2} dV \\
&=& \int \left( \frac{1}{3+8t} \langle \nabla |Ric|^{2}, \nabla R \rangle + \frac{2t}{3+8t} R|\nabla R|^{2} \right) dV\\
&=& \int \left( \frac{1}{3+8t} \langle \nabla |E|^{2}, \nabla R \rangle + \frac{2+6t}{3(3+8t)} R|\nabla R|^{2} \right) dV\,.
\end{eqnarray*}
Putting all together, we have showed that
$$
\int E(\nabla R, \nabla R) dV \,\,\leq\,\, \int \left( \frac{2}{3(3+8t)} \langle \nabla |E|^{2}, \nabla R \rangle - \frac{5+12t}{9(3+8t)}R |\nabla R|^{2} \right) dV \,.
$$
We are now in the position to prove the lemma. We want to estimate the following quantity

$$
\int \left( \frac{1}{2}\langle \nabla |E|^{2}, \nabla R \rangle + R|\nabla E|^{2} - \frac{1+2t}{6}R|\nabla R|^{2} - (1+2t)E(\nabla R, \nabla R) \right) dV \,.
$$
Since $t\geq -1/3$, from the previous inequality, we obtain

\begin{eqnarray*}
&&\int \left( \frac{1}{2}\langle \nabla |E|^{2}, \nabla R \rangle + R|\nabla E|^{2} - \frac{1+2t}{6}R|\nabla R|^{2} - (1+2t)E(\nabla R, \nabla R) \right) dV \\
&& \hspace{1cm}\geq \, \int \left( \frac{5+16t}{6(3+8t)} \langle \nabla |E|^{2}, \nabla R \rangle + R|\nabla E|^{2} + \frac{1+2t}{18(3+8t)}R |\nabla R|^{2} \right) dV \,.
\end{eqnarray*}
Now, since $R > 0$, then from Cauchy-Schwartz and Kato inequalities,  we have
\begin{eqnarray*}
\frac{5+16t}{6(3+8t)} \langle \nabla |E|^{2}, \nabla R \rangle &\geq & -\frac{|5+16t|}{3(3+8t)}|E||\nabla E| |\nabla R| \\
&\geq & - \frac{|5+16t|}{3(3+8t)} \eps R |\nabla R|^{2} - \frac{|5+16t|}{12(3+8t)\eps} \frac{|E|^{2}}{R} |\nabla E|^{2} \,,
\end{eqnarray*}
for every $\eps >0$. Choosing $\eps = (1+2t)/(6|5+16t|)$, we have 
$$
\frac{5+16t}{6(3+8t)} \langle \nabla |E|^{2}, \nabla R \rangle \,\geq \, - \frac{1+2t}{18(3+8t)} R |\nabla R|^{2} - \frac{(5+16t)^{2}}{2(1+2t)(3+8t)} \frac{|E|^{2}}{R} |\nabla E|^{2} \,.
$$
Hence, the integrand we want to estimate is bounded by
$$
\frac{5+16t}{6(3+8t)} \langle \nabla |E|^{2}, \nabla R \rangle + R|\nabla E|^{2} + \frac{1+2t}{18(3+8t)}R |\nabla R|^{2} \geq \frac{|\nabla E|^{2}}{R} \left( R^{2} - \frac{(5+16t)^{2}}{2(1+2t)(3+8t)} |E|^{2} \right) \,.
$$
Finally, using the pinching assumption~\eqref{pinch3d}, it is easy to prove that the right-hand side is non-negative, since
$$
\frac{(1+6t)^{2}}{24} \, \leq \, \frac{2(1+2t)(3+8t)}{(5+16t)^{2}} \,,
$$
if $-1/3 \leq t < -1/6$. This concludes the proof of the lemma.
\end{proof}

Combining this latter with equation~\eqref{eq-127} and inequality~\eqref{eq-128}, under the assumption of Theorem~\ref{t-three}, we have showed that
$$
\int R|E|^{2} \left( \frac{1+6t}{3}R - \frac{4}{\sqrt{6}}|E| \right) dV = 0 \,.
$$
Hence, $E\equiv 0$ on $M^{3}$ since $R> 0$ and the pinching assumption~\eqref{pinch3d} holds.

This concludes the proof of Theorem~\ref{t-three}.

\

\section{The Euler-Lagrange equation for $\Fts$} \label{s-total}

In this section we will compute the Euler-Lagrange equation satisfied by critical metrics for 
$$
\Fts(g) \,=\, \int |Ric|^{2} dV +  t \int R^{2}  dV + s \int |Rm|^{2} dV\,,
$$
on $\M1(M^{n})$. As we have already observed in the introduction, this functional substantially differs from $\Ft$ only in dimension greater than four. 

To compute the Euler-Lagrange equation for $\Fts$, we follow the computations in Section~\ref{s-euler}. The gradients of the functionals $\mathcal{\rho}$, $\mathcal{S}$ and $\mathcal{R}$ are given by (see~\cite[Proposition 4.66]{besse} and~\cite[Proposition 4.70]{besse})
$$
(\nabla \mathcal{\rho})_{ij} \,=\, -\Delta R_{ij} + \nabla^{2}_{ij} R - \frac{1}{2}(\Delta R)g_{ij} - 2 R_{ikjl}R_{kl} + \frac{1}{2} |Ric|^{2} g_{ij} \,,
$$

$$
(\nabla \mathcal{S})_{ij} \,=\, 2 \nabla^{2}_{ij} R - 2(\Delta R) g_{ij} -2 R R_{ij} + \frac{1}{2} R^{2} g_{ij} \,,
$$
and
$$
(\nabla \mathcal{R})_{ij} \,=\, -4\Delta R_{ij}+2\nabla^{2}_{ij} R - 2 R_{ikpq}R_{jkpq} + \frac{1}{2} |Rm|^{2}g_{ij} -4R_{ikjl}R_{kl} + 4 R_{ik}R_{jk} \,.
$$
Hence, the gradient of $\Fts$ reads 
\begin{eqnarray*}
(\nabla \Fts)_{ij} &=& -(1+4s)\Delta R_{ij} +(1+2t+2s)\nabla^{2}_{ij} R - \frac{1+4t}{2}(\Delta R)g_{ij} - 2 R_{ikjl}R_{kl} -2t R R_{ij}  \\
&& + \frac{1}{2} \Big( |Ric|^{2}+ t R^{2}+s|Rm|^{2}\Big) g_{ij}  - 2s R_{ikpq}R_{jkpq}  - 4s R_{ikjl}R_{kl} + 4s R_{ik}R_{jk}\,.
\end{eqnarray*}
Moreover, $g$ is critical for $\Fts$ on $\M1(M^{n})$ if and only if $(\nabla \Fts) = c \, g$, for some $c\in \RR$. Tracing this equation, we obtain
$$
\frac{n-4}{2}\Big(|Ric|^{2}+ t R^{2}+s|Rm|^{2}\Big)-\frac{n+4(n-1)t+4s}{2} \Delta R \,=\, nc \,.
$$
From these, we get that $g$ is critical for $\Fts$ on $\M1(M^{n})$ if and only if
\begin{eqnarray}\label{eq-rics}
\nonumber &&\hspace{-1cm}-(1+4s)\Delta R_{ij} + (1+2t+2s) \nabla^{2}_{ij} R -\frac{2t-2s}{n}(\Delta R) g_{ij} + \frac{2}{n}\Big(|Ric|^{2}+ t R^{2}+s|Rm|^{2}\Big)g_{ij}+\\&&- 2(1+2s) R_{ikjl}R_{kl} -2t R R_{ij} -2s R_{ikpq}R_{jkpq}+4sR_{ik}R_{jk} \,\,=\,\, 0\,,
\end{eqnarray}
coupled with the scalar equation
$$
\Big(n+4(n-1)t+4s\Big)\Delta R \,=\, (n-4) \Big( |Ric|^{2} + t R^{2}+s|Rm|^{2} - \lambda \Big) \,,
$$
where $\lambda =  \Fts (g)$. Substituting 
$E_{ij} = R_{ij} - \frac{1}{n} R g_{ij}$, we obtain the Euler-Lagrange equation of critical metrics for $\Fts$ on $\M1(M^{n})$.

\begin{prop}\label{p-euler}
Let $M^{n}$ be a closed manifold of dimension $n\geq 3$. A metric $g$ is critical for $\Fts$ on $\M1(M^{n})$ if and only if it satisfies the following equations

\begin{eqnarray}\label{eq1s}
(1+4s)\Delta E_{ij} &=&  (1+2t+2s) \nabla^{2}_{ij} R - \frac{1+2t+2s}{n} (\Delta R) g_{ij} - 2(1+2s) R_{ikjl} E_{kl} \\ \nonumber
&&-\frac{2+2nt-4s}{n} R E_{ij} + \frac{2}{n} \Big( |E|^{2}+s|Rm|^{2}\Big) g_{ij} -2sR_{ikpq}R_{jkpq} +4sE_{ik}E_{jk} \,,
\end{eqnarray}

\begin{equation}\label{eq2s}
\Big(n+4(n-1)t+4s\Big)\Delta R \,\,=\,\, (n-4) \Big( |Ric|^{2} + t R^{2}+s|Rm|^{2} - \lambda \Big) \,,
\end{equation}
where $\lambda = \Fts (g)$.
\end{prop}

In particular, any Einstein critical metric must satisfy the following pointwise condition (see also~\cite[Corollary 4.67]{besse})

\begin{cor} An Einstein metric is critical for $\Fts$ on $\M1(M^{n})$ if and only if it satisfies
$$
R_{ikpq}R_{jkpq} \,=\, \frac{1}{n} |Rm|^{2} g_{ij} \,.
$$
\end{cor}

\begin{cor} Any space form metric is critical for $\Fts$ on $\M1(M^{n})$.
\end{cor}

In the case $n\neq 4$ and 
$$
s\,=\, -\frac{n+4(n-1)t}{4} \,
$$
one has the constancy of the following quantity
$$
|Ric|^{2} + t R^{2}-\frac{n+4(n-1)t}{4}|Rm|^{2} \,.
$$
Since the norm of the Riemann curvature tensor is given by
$$
|Rm|^{2} \,=\, |W|^{2} + \frac{4}{n-2}|Ric|^{2} - \frac{2}{(n-1)(n-2)} R^{2} \,,
$$
we obtain the following 
\begin{cor} \label{c-yamabe}
Let $M^{n}$ be a closed manifold of dimension $n > 4$. If $g$ is a critical metric $g$ for $\mathcal{F}_{t,-\frac{n+4(n-1)t}{4}}$ on $\M1(M^{n})$, then the quantity 
$$
-\frac{n+4(n-1)t}{4} |W|^{2} + 4(n-2)\big(1+2(n-1)t\big) \sigma_{2} (A)
$$
is constant on $M^{n}$.
\end{cor}

Moreover, we notice that, when $t=s=-1/4$, Corollary~\ref{c-yamabe} applies, In this case, the integrand of the curvature functional $\mathcal{F}_{-\frac{1}{4},-\frac{1}{4}}$ vanishes if $n=3$ and it corresponds (in fact, it is proportional) to the Gauss-Bonnet integrand if $n=4$. Furthermore, it follows from equation~\eqref{eq1s} that all the second order terms in the Euler-Lagrange equation vanish. More precisely, one has the following remarkable fact, which, in part, was already observed by Berger in~\cite[Section 7]{berger3} (see also~\cite{labbi}).

\begin{cor} 
Let $M^{n}$ be a closed manifold of dimension $n > 4$. A metric $g$ is critical for $\mathcal{F}_{-\frac{1}{4},-\frac{1}{4}}$ on $\M1(M^{n})$ if and only if it satisfies the following equation.

\begin{eqnarray*}
-  R_{ikjl} E_{kl} +\frac{n-6}{2n} R E_{ij} + \frac{1}{2n} \Big( 4|E|^{2} - |Rm|^{2}\Big) g_{ij} +\frac{1}{2} R_{ikpq}R_{jkpq} - E_{ik}E_{jk} \, = \, 0 \,.
\end{eqnarray*}
Moreover, the quantity
$$
|W|^{2} + 2(n-2)(n-3) \sigma_{2}(A) 
$$
is constant on $M^{n}$.
\end{cor} 
As already suggested by Berger, it will be interesting to have a complete classification of these critical metrics. For the sake of completeness, we compute the pointwise and integral Weitzenb\"ock formulas of critical metrics for $\Fts$.

\begin{prop}
Let $M^{n}$ be a closed manifold of dimension $n > 4$. If $g$ is a critical metric $g$ for $\Ft$ on $\M1(M^{n})$, then the following formula holds

\begin{eqnarray}\label{eq-w}
\frac{1+4s}{2}\Delta |E|^{2} &=&  (1+4s) |\nabla E|^{2}+(1+2t+2s) E_{ij} \,\nabla^{2}_{ij} R - 2(1+2s) R_{ikjl} E_{ij} E_{kl} \\ \nonumber
&&-\frac{2+2nt-4s}{n} R |E|^{2}  -2s E_{ij}R_{ikpq}R_{jkpq}+4sE_{ij}E_{ik}E_{jk} \,.
\end{eqnarray}
\end{prop}

Integrating by parts and using second Bianchi identity, we obtain
\begin{cor}\label{cor-int}
Let $M^{n}$ be a closed manifold of dimension $n > 4$. If $g$ is a critical metric for $\Fts$ on $\M1(M^{n})$, then
\begin{eqnarray*}
&&\int \left((1+4s)|\nabla E|^{2} - \frac{(n-2)(1+2t+2s)}{2n}|\nabla R|^{2} \right) dV \,=\,\\
&&\,2 \int \left( (1+2s)R_{ikjl} E_{ij} E_{kl} + \frac{1+nt-2s}{n} R|E|^{2} +s E_{ij}R_{ikpq}R_{jkpq}-2sE_{ij}E_{ik}E_{jk}\right
) dV \,.
\end{eqnarray*}
In particular, if $g$ is a critical metric for $\mathcal{F}_{t,-\frac{n+4(n-1)t}{4}}$ on $\M1(M^{n})$, then
\begin{eqnarray*}
&&-(n-1)(1+4t)\int \left(|\nabla E|^{2} - \frac{(n-2)^{2}}{4n(n-1)}|\nabla R|^{2} \right) dV \,=\,\\
&&\, 2\int \Big(-\frac{n-2+4(n-1)t}{2}R_{ikjl} E_{ij} E_{kl} +\frac{n+2+2(3n-2)t}{2n} R|E|^{2}+ \\
&&\,-\frac{n+4(n-1)t}{4} E_{ij}R_{ikpq}R_{jkpq} + \frac{n+4(n-1)t}{2}E_{ij}E_{ik}E_{jk}\Big
) dV \,.
\end{eqnarray*}
\end{cor}

%\
%
%\section{A non-Einstein critical metric of $\Ft$} \label{s-example}
%
%[Proof of Theorem \ref{nonit}.]
%
%
%As in \cite{Lamontagne}, we consider the Lie group $SU(2)$, and choose the following
%basis of the Lie algebra $\mathfrak{su}(2)$
%\begin{align*}
%\left\{
%\left(
%\begin{matrix}
%i & 0\\
%0 & -i\\
%\end{matrix}
%\right) ,
%\left(
%\begin{matrix}
%0 & i\\
%i & 0 \\
%\end{matrix}
%\right) ,
%\left(
%\begin{matrix}
%0 & -1\\
%1 & 0 \\
%\end{matrix}
%\right)
%\right\}.
%\end{align*}
%Let $\{e_1, e_2, e_3 \}$ denote the corresponding basis
%of left-invariant vector fields. We declare these to
%be orthogonal with $e_1$ and $e_2$ of length $1$,
%and $e_3$ of length $s$ for positive $s \in \mathbf{R}$.
%The resulting metric is known as a Berger sphere, and
%for $s = 1$, is the round metric $g_S$.
%It follows that $\{\tilde{e}_1, \tilde{e}_2, \tilde{e}_3 \}
%= \{e_1, e_2, \frac{1}{s}e_3 \}$ is an orthonormal frame
%satisfying the commutation relations
%\begin{align*}
%\newcommand{\te}{\tilde{e}}
%[ \te_1, \te_2] = 2s \te_3,\
%[ \te_2, \te_3] = \frac{2}{s} \te_1,\
%[ \te_3, \te_1] = \frac{2}{s} \te_2.
%\end{align*}
%From these relations, it easily follows that in this basis
%the Ricci tensor is diagonal with eigenvalues
%$\{ 4 - 2s^2, 4-2s^2, 2s^2 \}$. We then have
%\begin{align*}
%|Ric|^2(s) &= 32 - 32s^2 + 12s^4,\\
%R^2(s) &= 64 - 32s^2 + 4s^4,
%\end{align*}
%and therefore
%\begin{align*}
%|Ric|^2 + t R^2 = 32(1 + 2 t) - 32(1 +t) s^2 + 4 (3 + t) s^4.
%\end{align*}
%Consequently,
%\begin{align*}
%\Ft[g_s] = c t^{4/3} (  32(1 + 2t) - 32(1 +t) s^2 + 4 (3 + t) s^4),
%\end{align*}
%for some constant $c > 0$.
%
%

\

\

\

\begin{ackn} The author is members of the Gruppo Nazionale per
l'Analisi Matematica, la Probabilit\`{a} e le loro Applicazioni (GNAMPA) of the Istituto Nazionale di Alta Matematica (INdAM).
\end{ackn}

\

\bibliographystyle{amsplain}
\bibliography{biblio}

\providecommand{\bysame}{\leavevmode\hbox to3em{\hrulefill}\thinspace}
\providecommand{\MR}{\relax\ifhmode\unskip\space\fi MR }
\providecommand{\MRhref}[2]{%
  \href{http://www.ams.org/mathscinet-getitem?mr=#1}{#2}
}
\providecommand{\href}[2]{#2}
\begin{thebibliography}{10}

\bibitem{ander2}
M.~T. Anderson, \emph{Extrema of curvature functionals on the space of metrics
  on {$3$}-manifolds}, Calc. Var. Partial Differential Equations \textbf{5}
  (1997), no.~3, 199--269.

\bibitem{ander3}
\bysame, \emph{Extrema of curvature functionals on the space of metrics on
  3-manifolds. {II}}, Calc. Var. Partial Differential Equations \textbf{12}
  (2001), no.~1, 1--58.

\bibitem{berger3}
M.~Berger, \emph{Quelques formules de variation pour une structure
  riemannienne}, Ann. Sci. \'Ecole Norm. Sup. (4) \textbf{3} (1970), 285--294.

\bibitem{besse}
A.~L. Besse, \emph{Einstein manifolds}, Springer--Verlag, Berlin, 2008.

\bibitem{cat2}
G.~Catino, \emph{Critical metric of the {$L^{2}$}-norm of the scalar
  curvature}, arXiv preprint server -- http://arxiv.org, to appear on Proc.
  Amer. Math. Soc., 2012.

\bibitem{gurvia3}
M.~J. Gursky and J.~A. Viaclovsky, \emph{A new variational characterization of
  three-dimensional space forms}, Invent. Math. \textbf{145} (2001), no.~2,
  251--278.

\bibitem{gurvia1}
\bysame, \emph{Rigidity and stability of {E}instein metrics for quadratic
  curvature functionals}, arXiv preprint server -- http://arxiv.org, to appear
  on J. Reine Angew. Math., 2011.

\bibitem{gurvia2}
\bysame, \emph{Critical metrics on connected sums of {E}instein
  four-manifolds}, arXiv preprint server -- http://arxiv.org, 2013.

\bibitem{hilbert}
D.~Hilbert, \emph{Die {G}rundlagen der {P}hysik}, Ann. Sci. \'Ecole Norm. Sup.
  (4) (1915), 461--472.

\bibitem{huli}
Z.~Hu and H.~Li, \emph{A new variational characterization of $n$-dimensional
  space forms}, Trans. Amer. Math. Soc. \textbf{356} (2003), no.~8, 3005--3023.

\bibitem{udo}
Z.~Hu, S.~Nishikawa, and U.~Simon, \emph{Critical metrics of the {S}chouten
  functional}, J. Geom. \textbf{98} (2010), no.~1-2, 91--113.

\bibitem{labbi}
M.-L. Labbi, \emph{Variational properties of the {G}auss-{B}onnet curvatures},
  Calc. Var. Partial Differential Equations \textbf{32} (2008), no.~2,
  175--189.

\bibitem{lamontagne2}
F.~Lamontagne, \emph{Une remarque sur la norme {$L^2$} du tenseur de courbure},
  C. R. Acad. Sci. Paris S\'er. I Math. \textbf{319} (1994), no.~3, 237--240.

\bibitem{lamontagne1}
\bysame, \emph{A critical metric for the {$L^2$}-norm of the curvature tensor
  on {$S^3$}}, Proc. Amer. Math. Soc. \textbf{126} (1998), no.~2, 589--593.

\bibitem{smot}
N.~K. Smolentsev, \emph{Spaces of {R}iemannian metrics}, J. Math. Sci.
  \textbf{142} (2007), no.~5, 2436--2519.

\bibitem{tanno}
S.~Tanno, \emph{Deformations of {R}iemannian metrics on 3-dimensional
  manifolds}, T\^ohoku Math. J. (2) \textbf{27} (1975), no.~3, 437--444.

\bibitem{via2}
J.~A. Viaclovsky, \emph{Conformal geometry, contact geometry, and the calcu-
  lus of variations}, Duke Math. J. \textbf{101} (2000), no.~2, 283--316.

\bibitem{via1}
\bysame, \emph{The mass of the product of spheres}, arXiv preprint server --
  http://arxiv.org, 2013.

\end{thebibliography}

\bigskip

\parindent=0pt

\

\end{document}